\documentclass[11pt]{amsart}
\usepackage{amsmath, amsthm, mathabx,amscd, amsfonts, amssymb, hyperref, mathrsfs, textcomp, epsfig, a4wide, graphicx, IEEEtrantools, tikz, verbatim, xypic }
\usepackage[shortlabels]{enumitem}

\usepackage[all,cmtip]{xy}
\usetikzlibrary{matrix, decorations.pathreplacing,angles,quotes,cd}

\newtheorem{prop}{Proposition}[section]
\newtheorem{thm}[prop]{Theorem}
\newtheorem{lemma}[prop]{Lemma}
\newtheorem{cor}[prop]{Corollary}
\theoremstyle{definition}
\newtheorem{defn}[prop]{Definition}

\usepackage{mathtools}

\newcommand{\spin}{\mathfrak{s}}

     \RequirePackage{rotating}                   
    \def\HMt{%
       \setbox0=\hbox{$\widehat{\mathit{HM}}$}
       \setbox1=\hbox{$\mathit{HM}$}
       \dimen0=1.1\ht0
       \advance\dimen0 by 1.17\ht1
       \smash{\mskip2mu\raise\dimen0\rlap{%
          \begin{turn}{180}
              {$\widehat{\phantom{\mathit{HM}}}$}
           \end{turn}} \mskip-2mu    
                \mathit{HM}
    }{\vphantom{\widehat{\mathit{HM}}}}{}}
    
         \RequirePackage{rotating}                   
    \def\HSt{%
       \setbox0=\hbox{$\widehat{\mathit{HS}}$}
       \setbox1=\hbox{$\mathit{HS}$}
       \dimen0=1.1\ht0
       \advance\dimen0 by 1.17\ht1
       \smash{\mskip2mu\raise\dimen0\rlap{%
          \begin{turn}{180}
              {$\widehat{\phantom{\mathit{HS}}}$}
           \end{turn}} \mskip-2mu    
                \mathit{HS}
    }{\vphantom{\widehat{\mathit{HS}}}}{}}

    \newcommand{\HMb}{\overline{\mathit{HM}}}
        \newcommand{\HMr}{{\mathit{HM}}}
\newcommand{\HMf}{\widehat{\mathit{HM}}}

\theoremstyle{remark}
\newtheorem{remark}{Remark}[section]

\newtheorem{ass}{Assumption}[section]

\begin{document}
\title{Topology of the Dirac equation on spectrally large three-manifolds}

\author{Francesco Lin}
\address{Department of Mathematics, Columbia University} 
\email{flin@math.columbia.edu}

\begin{abstract}

The interaction between spin geometry and positive scalar curvature has been extensively explored. In this paper, we instead focus on Dirac operators on Riemannian three-manifolds for which the spectral gap $\lambda_1^*$ of the Hodge Laplacian on coexact $1$-forms is large compared to the curvature. As a concrete application, we show that for any spectrally large metric on the three-torus $T^3$, the locus in the torus of flat $U(1)$-connections where (a small generic pertubation of) the corresponding twisted Dirac operator has kernel is diffeomorphic to a two-sphere.
\par
While the result only involves linear operators, its proof relies on the non-linear analysis of the Seiberg-Witten equations. It follows from a more general understanding of transversality in the context of the monopole Floer homology of a torsion spin$^c$ three-manifold $(Y,\spin)$ with a large spectral gap $\lambda_1^*$. When $b_1>0$, this gives rise to a very rich setup and we discuss a framework to describe explicitly in certain situations the Floer homology groups of $(Y,\spin)$ in terms of the topology of the family of Dirac operators parametrized by the torus of flat $U(1)$-connections on $Y$.
\end{abstract}

\maketitle

\section*{Introduction}

The interplay between the geometry of Dirac operators and scalar curvature has been intensively explored since the foundational work of Lichnerowicz and Atiyah-Singer (see for example \cite{Roe} for an introduction). A cornerstone in the theory is Gromov and Lawson's proof that the torus $T^n$ does not admit metrics of positive scalar curvature \cite{GL}. In particular, their approach involves studying the family of Dirac operators $\{D_B\}$ obtained from the spin Dirac operator by twisting by a flat $U(1)$-connection.
\\
\par
In three-dimensions, the topology of closed orientable manifolds admitting a metric of positive scalar curvature is in fact \textit{very} restricted: any such manifold is a connected sum of spherical space forms and $S^1\times S^2$s (\cite{Per}, see also \cite{MT}). Because of this, it is natural to ask whether there are interactions between Dirac operators and other natural geometric notions.
\par
In this paper, we will explore the relation between Dirac operators and the spectral geometry on three-manifolds. More specifically, we will consider for a closed Riemannian three-manifold $(Y,g)$ the first eigenvalue $\lambda_1^*>0$ of the Hodge Laplacian $\Delta=(d+d^*)^2$ acting on coexact $1$-forms $d^*\Omega^2$, and consider the situation in which this spectral gap is large in the following sense.
\begin{defn}\label{speclarge}
The Riemannian $3$-manifold $(Y,g)$ is \textit{spectrally large} if the inequality
\begin{equation}
\lambda_1^*>-\frac{1}{2}\mathrm{inf}_{p\in Y}\tilde{s}(p)
\end{equation}
holds, were we denote by $\tilde{s}(p)$ the sum of the two least eigenvalues of the Ricci curvature at $p$.
\end{defn}
In particular, for a metric satisfying $\mathrm{Ricci}\geq -k$ for some $k\geq 0$ everywhere, if $\lambda_1^*>k$ then $(Y,g)$ is spectrally large. As a basic example, any flat metric on the three-torus $T^3$ is spectrally large; furthermore, the condition is open in the $C^2$-topology in the space of metrics hence there are plenty of non-flat metrics on $T^3$ which are spectrally large. Of course, determining whether a given metric on $T^3$ is spectrally large and understanding the basic properties of the space of such metrics are both very challenging questions.
\par
With this in mind, we will begin by focusing on applications of our techniques to the study of the Dirac equation on a general spectrally large $(T^3,g)$. Denote the spinor bundle (which is a rank $2$ hermitian bundle) by $S\rightarrow T^3$. For a fixed section $A\in i\Gamma(\mathfrak{u}(S))$ of the bundle of (complex-linear) self-adjoint endomorphism of $S$ (which we assume for simplicity to be smooth), we will consider the family of perturbed Dirac equations
\begin{equation}\label{perturbeddirac}
D_B\Psi =A\Psi
\end{equation}
parametrized by the torus $\mathbb{T}_{T^3}$ of spin$^c$ connections on $S$ (up to gauge) for which the induced $U(1)$-connection $B^t$ on $\mathrm{det}(S)$ is flat. We can identify the latter space with
\begin{equation*}
\mathbb{T}_{T^3}=\frac{H^1(T^3;i\mathbb{R})}{H^1(T^3;2\pi i\mathbb{Z})},
\end{equation*}
which is a three-dimensional torus itself. We then have the following.
\begin{thm}\label{main}
Let $g$ be a spectrally large metric on $T^3$. For a generic $C^0$-small $A\in \Gamma( i\mathfrak{u}(S))$, the locus $\mathsf{K}$ consisting of $[B]\in\mathbb{T}_{T^3}$ for which the perturbed Dirac equation (\ref{perturbeddirac}) has non-zero solutions is diffeomorphic to a two-sphere.
\end{thm}

Notice that some genericity assumptions are necessary. For example for a flat metric $g_{\mathrm{flat}}$ and $A\equiv 0$, the Lichnerowicz formula implies that a harmonic spinor is necessarily parallel, and one checks directly that there is exactly one point in $\mathbb{T}_{T^3}$ for which (\ref{perturbeddirac}) has solutions.
\par
In this setup, we can also verify explicitly that the statement of Theorem \ref{main} holds for $A$ constant multiple of the identity $\delta\neq 0$ small. This explicit computation is in fact the main inspiration behind our result, which extends its conclusion to the setup of spectrally large metrics where a direct analysis is not available.
\begin{remark}
For \textit{any} choice of metric $g$ and $A$ the exists $[B]\in\mathbb{T}_{T^3}$ such that (\ref{perturbeddirac}) has non-trivial solutions, so that $\mathsf{K}$ is always non-empty. This is a consequence of the Atiyah-Singer index theorem for families (see \cite{ASskew} and \cite{APS3}), and the non-triviality of the triple cup product on $T^3$. Of course, in general the locus $\mathsf{K}$ for which (\ref{perturbeddirac}) has solutions will be very complicated.
\end{remark}
Under some mild additional assumptions, we can say more about the location of the two-sphere $\mathsf{K}=S^2$ inside $\mathbb{T}_{T^3}$. First of all, because $\mathbb{T}_{T^3}$ is an irreducible three-manifold, Alexander's theorem implies that $\mathsf{K}$ separates $T^3$ in two components, one of which is diffeomorphic to the interior of an embedded ball $B^3\hookrightarrow T^3$ (see for example \cite{Rol}). In $\mathbb{T}_{T^3}$ there are eight special points corresponding to the spin structures $\mathsf{s}$ on $T^3$; once one is fixed, the other ones are obtained by twisting by a flat line bundle with holonomy $\pm1$ around each loop. Among these spin structures there is a special one $\mathsf{s}_0$, which can identified by the fact that it has Rokhlin invariant $\mu(T^3,\mathsf{s}_0)=1$, while the other ones have $\mu(T^3,\mathsf{s})=0$ (see \cite[Ch. 5]{Kir}). We then have the following.
\begin{thm}\label{mainspin}
In the setup the Theorem \ref{main}, assume that under the natural decomposition
\begin{equation*}
A=A_0+\tau\cdot I
\end{equation*}
with $A_0$ traceless and $\tau=\mathrm{tr}A/2$ a function, the $C^0$ norm of $A_0$ is small compared to that of $\tau$. Then the component of $\mathbb{T}_{T^3}\setminus \mathsf{K}$ which is diffeomorphic to $\mathrm{int} B^3$ contains the spin structure $\mathsf{s}_0$, while the other component contains all the remaining spin structures $\mathsf{s}\neq \mathsf{s}_0$.
\end{thm}
Notice that in the case of a flat metric $g_{\mathrm{flat}}$, the spin structure $\spin_0$ can be identified as the unique $B_0$ for which $D_{B_0}$ has non-trivial kernel, and the conclusion of the result can be again checked explicitly for $A=\delta\neq 0$ small.
\\
\par
While the statement and conclusion of Theorems \ref{main} and \ref{mainspin} only involve linear operators, their proof relies on the non-linear analysis of the Seiberg-Witten equations on $(T^3,g)$. More generally, we will consider a Riemannian three-manifold $(Y,g)$ equipped with a torsion spin$^c$ structure $\spin$ (i.e. the spinor bundle has torsion first Chern class $c_1$). For a pair $(B,\Psi)$ consisting of a general spin$^c$ connection $B$ on the spinor bundle $S$ (i.e., the twisting connection $B^t$ is non necessarily flat) and spinor $\Psi$, the unperturbed version of the equations are given by
\begin{align}\label{seibergwitten}
\begin{split}
D_B\Psi&=0\\
\frac{1}{2}\rho(\ast F_{B^t})+(\Psi\Psi^*)_0&=0.
\end{split}
\end{align}
Here $\rho$ denotes the Clifford multiplication, $F_{B^t}$ is the curvature of $B^t$, and $(\Psi\Psi^*)_0$ is the traceless part of the endomorphism $\Psi\Psi^*$ of $S$. These are the equations for the critical points of the Chern-Simons-Dirac functional
\begin{equation*}
\mathcal{L}(B,\Psi)=-\frac{1}{8}\int_{Y}(B^t-B_0^t)\wedge(F_{B^t}+F_{B^t_0})+\frac{1}{2}\int_{Y} \langle D_B\Psi,\Psi\rangle d\mathrm{vol},
\end{equation*}
where $B_0$ is a fixed reference connection (which we can assume for simplicity to have $F_{B_0^t}=0$).
In \cite{KM}, the authors introduce the monopole Floer homology group $\HMt_*(Y,\spin)$, a topological invariant (hence independent of the metric) obtained by suitably counting solutions to these equations in way formally analogous to the $S^1$-equivariant Morse homology of $\mathcal{L}$. 
\begin{remark}
To simplify the discussion, when discussing Floer homology we will use $\mathbb{Q}$ coefficients throughout the paper (and keep it implicit in our notation).
\end{remark}
The starting point of our analysis are the results from \cite{LinSpec} and \cite{LL} stating that when $(Y,g)$ is spectrally large, the set of solutions to (\ref{seibergwitten}) up to gauge consists exactly of configurations $(B,0)$ where $[B]\in\mathbb{T}_Y$. When $b_1(Y)=0$, the latter consists of a single point and one readily concludes (being careful with transversality) that $Y$ is an $L$-space (i.e. it has trivial reduced Floer homology for all spin$^c$ structures).
\par
The situation when $b_1>0$ is much richer and complicated because $\mathbb{T}_Y$ is now a critical submanifold of $\mathcal{L}$ hence in order to determine $\HMt_*(Y,\spin)$ one must always perturb the equations. Furthermore, $\mathbb{T}_Y$ is in general \textit{not} a Morse-Bott singularity, the Hessian of $\mathcal{L}$ being degenerate in the normal directions exactly where the Dirac operator $D_B$ has kernel. The key technical content of the paper will be then to provide a framework to perturb the equations that allows to achieve transversality while maintaining a detailed understanding of the solutions to the equations in order to obtain a somewhat concrete description of the chain complex computing $\HMt_*(Y,\spin)$. Notice that, unlike the case $b_1=0$, $\HMt_*(Y,\spin)$ will be highly dependent on the choice of torsion spin$^c$ structure.
\\
\par
A simple observation is that the family of equations (\ref{perturbeddirac}) for $[B]\in\mathbb{T}_Y$ naturally arises when considering perturbations of the Chern-Simons functional of the form
\begin{equation}\label{pertCSD}
\mathcal{L}(B,\Psi)-\frac{1}{2}\langle \Psi,A\Psi\rangle_{L^2},
\end{equation}
which we call of \textit{spinorial type}. As we will see, on a spectrally large manifold the corresponding perturbed Seiberg-Witten equations for small $A$ still have no irreducible solutions. Furthermore, for generic spinorial perturbations the locus $\mathsf{K}\subset\mathbb{T}_Y$ is a codimension one hypersurface which has singularities in codimension at least $4$ (corresponding to the locus where the kernel has multiplicity); in particular, when $b_1(Y)\leq 3$ it is smooth. To achieve transversality, we will then add additional perturbation arising from a suitable Morse function on $\mathbb{T}_Y$ adapted to $\mathsf{K}$.
\\
\par
It has to be noted that in general, even though one can describe the generators of the Floer chain complex explicitly in terms of the family of perturbed Dirac operators, the determination of the differential is not possible. Nevertheless, we will be able to provide an explicit computation of the invariants under the additional assumption that the spectral flow between any two points in $\mathbb{T}_Y\setminus \mathsf{K}$ is at most one as in the case of $(T^3,{g}_{\mathrm{flat}})$. We will say that a manifold with this property is \textit{of the simplest type} (see Definition \ref{simplestdef}). As a concrete example, we will also show that (for specific metrics modeled on $\mathbb{R}\times\mathbb{H}^2$) the manifolds $S^1\times \Sigma_h$ for $h=2,3$ are of the simplest type, and combining our methods with classical results in algebraic geometry we will recover the analogous computations obtained in Heegaard Floer homology by Ozsv\'ath-Szab\'o \cite{OS} and Jabuka-Mark \cite{JM} using surgery techniques. The famous curves of Bolza and Klein will play a central role in the process.
\par
With this in hand, the proof of the main result Theorem \ref{main} will follow by a detailed analysis of which submanifolds $\mathsf{K}\subset \mathbb{T}_{T^3}$ give rise to Floer chain complexes computing the correct homology group $\HMt_*(T^3,\spin_0)$ (where $\spin_0$ denotes the unique torsion spin$^c$ structure). In particular, the conclusion also holds for any torsion spin$^c$ three-manifold $(Y,\spin)$ with $b_1=3$ for which $\HMt_*(Y,\spin)\cong \HMt_*(T^3,\spin_0)$ (as absolutely graded $\mathbb{Q}[U]\otimes \Lambda^*(H_1/\mathrm{tors})$-modules, up to an overall grading shift). By contrast, the refined statement in Theorem \ref{mainspin} takes as input additional structure arising from spin topology of ${T}^3$.
\par
We conclude by remarking that the determination of $\HMt_*(T^3,\spin_0)$ has been on fundamental importance in Seiberg-Witten theory and its applications to four-dimensional topology \cite{MMS},\cite{Tau}.
\\
\par
\textit{Organization. }In Section \ref{review} we recollect some background material in monopole Floer homology and the geometry of Dirac operators necessary for our discussion. In Section \ref{chaincompl} we provide a description of the generators of the Floer chain complex of a spectrally large three-manifold in terms of Morse theory on the torus of flat connections. In Section \ref{simplestdef} we introduce a special class of three-manifolds, which we call of the simplest type, and provide a concrete computation of their Floer homology groups. In Section \ref{sigmag} we discuss in detail two interesting examples of manifolds of the simplest type, namely $S^1\times \Sigma_g$ for $g=2,3$. Finally, in Section \ref{T3} we prove the main results of the paper, Theorems \ref{main} and \ref{mainspin}. 
\\
\par
\textit{Acknowledgements. }Understanding the Floer homology of $S^1\times\Sigma$ from a geometric viewpoint was a key motivation behind the project, and the author would like to thank Tom Mrowka for suggesting to think about it about a decade ago. He also thanks the anonymous referee of his previous paper \cite{LinTheta} for some helpful comments. This work was partially supported by the Alfred P. Sloan Foundation and NSF grant DMS-2203498.

\vspace{0.3cm}

\section{Background material}\label{review}
We recall some essential ideas about monopole Floer homology; we refer the reader to \cite{KM} for the definitive reference, and \cite{LinLec} for a friendly introduction to the topic.
Fix a spin$^c$ $3$-manifold $(Y,\spin)$. Given a spin$^c$ connection $B$, the curvature $2$-form $F_{B^t}$ of the connection $B^t$ induced on the determinant line bundle is always closed. Throughout the paper, we will only consider torsion spin$^c$ structures; by Chern-Weil theory, this implies that $F_{B^t}$ is an \textit{exact} $2$-form. Because of this one sees that reducible solutions $(B,0)$ of the equations (\ref{seibergwitten}) up to gauge are in correspondence with the torus $\mathbb{T}_Y$ of spin$^c$ connections $B$ with $B^t$ flat, see \cite[Section 4.2]{KM}.
\par
In particular, when $b_1>0$, the unperturbed equations are \textit{never} transversely cut out. Furthermore, in order to define the monopole Floer homology groups one needs to work in the blown-up configurations space consisting of triples $(B,r,\psi)$ with $r\in\mathbb{R}^{\geq0}$ and $\|\psi\|_{L^2}=1$. In this setup, one needs to introduce suitable perturbations and transversality for the three-dimensional equations can be phrased in terms of the following criteria:
\begin{itemize}
\item the critical points of (the perturbation of ) $\mathrm{grad}\mathcal{L}$ in the blow-down are non-degenerate.
\item at each reducible critical point $(B,0)$ of the (perturbation of) $\mathrm{grad}\mathcal{L}$, the (perturbation of the) Dirac operator $D_B$ has simple spectrum and trivial kernel.
\end{itemize}
In \cite[Ch. 11]{KM} a large space of perturbations is introduced so that these condition can be achieved. Of course, even if one had a concrete understanding of the solutions to the unperturbed equations, this might not be the case anymore after appealing to the general results for achieving transversality. Our goal in this paper is to study a concrete instance in which one can both achieve transversality and mantain a concrete understanding of the space of solutions. For example, in the context of transversality for the flat torus \cite[Ch. 37]{KM} and mapping tori of cyclic coverings $\Sigma\rightarrow \mathbb{P}^1$ in \cite{LinTheta}, for $\delta,\varepsilon> 0$ perturbations of the Chern-Simons functional of the form
\begin{equation*}
\widetilde{\mathcal{L}}(B,\Psi)=\mathcal{L}(B,\Psi)-\frac{\delta}{2}\|\Psi\|^2+\varepsilon f(B)
\end{equation*}
were considered. Here for a function $f:\mathbb{T}_Y\rightarrow \mathbb{R}$ we define (abusing notation) $f(B)$ to be the evaluation at the harmonic part of $B-B_0$ (thought of as an element in $ H^1(Y,i\mathbb{R})$).
\par
In those situations, the analysis of tranversality begins by studying the perturbations for which $\varepsilon=0$. In this context, the first of the Seiberg-Witten equations (\ref{seibergwitten}) is replaced by the eigenvalue equation
\begin{equation}
D_B\Psi=\delta\Psi.
\end{equation}
In the level of generality we are interested in, we will need to introduce additional perturbations involving only spinors that generalize the ${\delta}\|\Psi\|^2/2$ term. In particular, we will consider for a section $A\in\Gamma( i\mathfrak{u}(S))$ the perturbed Chern-Simons-Dirac functional (\ref{pertCSD}), for which the corresponding perturbed Seiberg-Witten equations are
\begin{align}\label{pertseibergwitten}
\begin{split}
D_B\Psi&=A\Psi\\
\frac{1}{2}\rho(\ast F_{B^t})+(\Psi\Psi^*)_0&=0.
\end{split}
\end{align}
We will refer to these perturbations as \textit{of spinorial type}. Notice that they are not of cylinder type in the sense of \cite[Ch. 11]{KM}; nevertheless, it is readily checked that they define $k$-tame perturbations, hence can be used to perturbed the equations in the relevant functional setup. The main advantage of using them is that, while the first the Seiberg-Witten changes in an understandable way, the second equation is left unchanged. In particular, reducible solutions $(B,0)$ to the equations are still identified with the torus $\mathbb{T}_Y$.
\begin{remark}\label{alternative}
Under the decomposition $A=A_0+\tau\cdot I$ with $A_0$ traceless, we can rewrite (\ref{perturbeddirac}) as
\begin{equation*}
D_{\tilde{B}}\Psi=\tau\cdot \Psi
\end{equation*}
where $\tilde{B}=B-\rho^{-1}(A_0)$ in a new spin$^c$ connection. Hence we can interpret the family of equations (\ref{perturbeddirac}) for $[B]\in\mathbb{T}$ as a family of equations where $\tilde{B}$ varies in the torus of spin$^c$ connections with fixed curvature.
\end{remark}
Of course, in general the equations (\ref{pertseibergwitten}) will admit irreducible solutions $(B,\Psi)$ (i.e. $\Psi$ not identically zero), and it is in general quite challenging to understand explicitly whether they do. In \cite{LinSpec} and \cite{LL} it is proved that a suitable condition of spectral nature guarantees that there are no such solutions. While those results were focused towards the case $b_1=0$, we will now adapt the discussion to fit our needs.
\par
Recall that we denote by $\lambda_1^*$ the least eigenvalue of the Hodge Laplacian $\Delta=(d+d^*)^2$ on coexact $1$-forms, and by $\tilde{s}(p)$ the sum of the two least eigenvalues of the Ricci curvature at $p\in Y$, and we say that the metric $g$ is spectrally large if the inequality
\begin{equation*}
\lambda_1^*>-\frac{1}{2}\inf_{p\in Y} \tilde{s}(p)
\end{equation*}
holds (cf. Definition \ref{speclarge}).
\begin{prop}\label{noirred}
Suppose $(Y,g)$ is spectrally large. Then for a spinorial type perturbation $A\in \Gamma(i\mathfrak{u}(S))$ of sufficiently small $C^0$ norm, the $A$-perturbed Seiberg--Witten equations (\ref{pertseibergwitten}) do not admit irreducible solutions for all torsion spin$^c$ structures $\spin$ on $Y$.
\end{prop}
\begin{proof}
The way in which $\lambda_1^*$ comes into play in the proofs in \cite{LinSpec} and \cite{LL} is through its Rayleigh quotient interpretation:
\begin{equation}\label{rayleigh}
\|d\xi\|^2_{L^2}\geq \lambda_1^*\|\xi\|^2_{L^2}\text{ for any coexact $1$-form $\xi$.}
\end{equation} 
In our setup, $\ast F_{B^t}$ is coexact $1$-form, and the second Seiberg-Witten equation is unchanged, so the proof applies without changes. For completeness, let us point out how to deal with the perturbed operator $D_B-A$. The main effect is that the pointwise inner product $\langle \Psi,D_B^2\Psi\rangle$ does not vanish identically. In the argument one takes its integral after multiplying by $|\Psi|^2$; integrating by parts, we obtain
\begin{align*}
\int_Y|\Psi|^2\langle \Psi,D_B^2\Psi\rangle d\mathrm{vol}&=\int_Y\langle |\Psi|^2 \Psi,D_B^2\Psi\rangle d\mathrm{vol}=\\
&=\int_Y\langle D_B|\Psi|^2 \Psi,D_B\Psi\rangle d\mathrm{vol}\\
&=\int_Y\langle D_B|\Psi|^2 \Psi,A\Psi\rangle d\mathrm{vol}.
\end{align*}
We can then bound the integrand
\begin{equation*}
\lvert\langle D_B|\Psi|^2 \Psi,A\Psi\rangle\lvert\leq C_A |\nabla_B\Psi|\cdot|\Psi|^3\leq \frac{C_A}{2}\cdot(|\Psi|^4+|\nabla_B\Psi|^2|\Psi|^2)
\end{equation*}
where we used $2ab\leq a^2+b^2$ and $C_A$ is a constant depending only on the $C^0$ norm of $A$. The rest of the argument then is carried in exactly the same way without major changes.
\end{proof}

\begin{remark}
The key point is that a spinorial perturbation does not change the second equation (\ref{pertseibergwitten}) so that one can use the Rayleigh quotient interpretation of the first eigenvalue on coexact forms (\ref{rayleigh}). In general, the issue lies in controlling the perturbation in the direction of harmonic forms (the tangent space to $\mathbb{T}_Y$); we will discuss this in detail when adding Morse perturbations in Section \ref{chaincompl}.
\end{remark}

In the case $b_1=0$, the perturbed Dirac operator at the (unique up to gauge) flat connection is $D_{B}-\delta$, hence it has no kernel for generic $\delta$; furthermore, one can add additional small perturbations as in \cite[Ch. 12]{KM} to make its spectrum simple. Hence a space satisfying the assumption of Proposition \ref{noirred} is an $L$-space.
\par
By contrast when $b_1>0$, the $A$-perturbed equations still have $\mathbb{T}_Y$ as reducible critical set, so the set of solutions to (\ref{pertseibergwitten}) is \textit{never} transversely cut out in the sense of \cite{KM}. Furthermore, even though $\mathbb{T}_Y$ is a smooth manifold, it is in general not a Morse-Bott critical submanifold: this fails exactly on the locus $\mathsf{K}\subset\mathbb{T}_Y$ where $D_{B}-A$ has kernel. Our main goal is to extract information about the Floer homology groups of $(Y,\spin)$ in terms of the pair $(\mathbb{T}_Y,\mathsf{K})$.
\\
\par
In order to do so, we begin by reviewing some well-known facts on the space $\mathrm{Op}^{sa}$ of Dirac type operators of the form $D_{B_0}+h$ where $B_0$ is a fixed reference connection and $h$ is a $0$-th order perturbation. For the sake of the exposition, we will not be too detailed about the functionals setup, and refer the reader to \cite[Ch. 12]{KM} for details.
\par
We can consider the subspace of $\mathcal{K}\subset \mathrm{Op}^{sa}$ consisting of operators with non-trivial kernel. This admits a stratification
\begin{equation*}
\mathcal{K}=\coprod_{i>0}\mathcal{K}_i
\end{equation*}
where $\mathcal{K}_i$ is the subspace of operators with kernel having dimension exactly $i$. The latter is a smooth Banach submanifold of $\mathrm{Op}^{sa}$ of codimension $i^2$, where the normal bundle at an operator $T\in\mathcal{K}_i$ is naturally identified with the space of self-adjoint operators on $\mathrm{ker}T$:
\begin{equation}\label{normal}
N_T\mathcal{K}_i=i\mathfrak{u}(\mathrm{ker}T).
\end{equation}
Given this discussion, we can state the following.
\begin{lemma}\label{generic}
For generic spinorial perturbation $A\in\Gamma( i\mathfrak{u}(S))$, the locus $\mathsf{K}\subset \mathbb{T}_Y$ for which $D_B-A$ has kernel admits a stratification
\begin{equation*}
\mathsf{K}=\coprod_{i>0}\mathsf{K}_i
\end{equation*}
corresponding to the dimension of the kernel, and $\mathsf{K}_i\subset\mathbb{T}_Y$ is a smooth submanifold of codimension $i^2$.
\end{lemma}
\begin{proof}
This follows from a standard transversality argument. In particular (again being imprecise about the functional setup) we can consider the map
\begin{align*}
F:\mathbb{T}\times \Gamma(i\mathfrak{u}(S))&\rightarrow \mathrm{Op}^{sa}\\
(B,A)&\mapsto D_{B}-A.
\end{align*}
Because the normal bundle of each stratum is given by (\ref{normal}), we readily see that $F$ is transverse to all strata, hence by Sard-Smale for a generic $A\in \Gamma(i\mathfrak{u}(S))$ the map $F(-,A)$ is also transverse to all strata.
\end{proof}

While the discussion above is local in nature, for our purposes we will need to discuss some global aspects of the family of operators $\{D_B-A\}$. In our setup, the Atiyah-Singer index theorem for families \cite{APS3}, determines the classifying map
\begin{equation*}
\mathrm{ind}:\mathbb{T}_Y\rightarrow U(\infty)
\end{equation*}
of the family up to homotopy in terms of the triple cup-product
\begin{equation*}
\cup_Y^3:\Lambda^3H^1(Y;\mathbb{Z})\rightarrow \mathbb{Z},
\end{equation*}
thought of as an element in
\begin{equation}\label{cup3}
\cup_Y^3\in H^3(\mathbb{T}_Y;\mathbb{Z})\cong \Lambda^3 H^1(\mathbb{T}_Y;\mathbb{Z})\cong \Lambda^3 H^1(Y;\mathbb{Z})^*.
\end{equation}
We refer the reader to \cite{LinDir} for a quick introduction to the topic geared towards our needs. We have the following.

\begin{lemma}\label{ktop}
The hypersurface $\mathsf{K}\subset \mathbb{T}_Y$ (which is possibly singular and disconnected) defines the zero class in $H_{n-1}(\mathbb{T}_Y;\mathbb{Z})$. The restriction of the family of operators $\{D_B-A\}$ to each component $U$ of $\mathbb{T}_Y\setminus\mathsf{K}$ is homotopically trivial. In particular, $\cup_Y^3\in H^3(\mathbb{T}_Y;\mathbb{Z})$ restricts to zero on each component $U$.
\end{lemma}
\begin{remark}
Notice that the top stratum $\mathsf{K}_1\subset\mathsf{K}$ is naturally oriented because its normal bundle can identified with the multiples of the identity.
\end{remark}
\begin{proof}
When $\spin$ is torsion, the index theorem implies that the family of operators $\{D_B-A\}$ has no spectral flow around loops; the first part then follows from the fact that the spectral flow around a loop $\gamma$ can be interpreted as the (well-defined) intersection number $\gamma\cdot \mathsf{K}$.
For the second part, by definition the restriction to $U$ gives rise to a family of operators with no kernel, hence invertible. We conclude because families of invertible operators are homotopically trivial (see \cite{ASskew}).
\end{proof}

\begin{remark}\label{double}
While the remainder of the paper we will only focus on the case in which $\mathsf{K}=\mathsf{K}_1$ is smooth, we will say a couple of words about the simplest situation in which the $\mathsf{K}$ has singularities, namely the case in which there are points for which $D_B-A$ has two-dimensional kernel. The locus where the kernel has dimension at least 3 has codimension $9$, so this is generically true for $b_1(Y)\leq 8$.
\par
We will denote the locus by $\mathsf{K}_2\subset \mathbb{T}$. By (\ref{normal}), the normal bundle at a point $T\in\mathsf{K}_2$ is naturally identified with $i\mathfrak{u}(\mathrm{ker}(T))$, the real $4$-dimensional space of hermitian operators on $\mathrm{ker}(T)$. Identifying $\mathrm{ker}(T)=\mathbb{C}^2$, we have
\begin{equation*}
i\mathfrak{u}(2)=\left\{
\begin{bmatrix}
a& z\\
\bar{z}& b
\end{bmatrix}\lvert a,b\in\mathbb{R},z\in\mathbb{C}.
\right\}
\end{equation*}
In this model, the locus of matrices with kernel is given by the quadric in $\mathbb{R}^4$
\begin{equation*}
Q=\{ab-|z|^2=0\},
\end{equation*}
which is a quadratic form with signature $(1,3)$, and the only matrix with multiplicity $2$ is the zero matrix. Topologically, $Q$ is the union over two disjoint copies of $S^2$; in general, when transversely cut out (cf. Lemma \ref{generic}), the locus in $\mathsf{K}$ where the dimension of the kernel is at most $2$ is a $Q$ bundle over $\mathsf{K}_2$ near $\mathsf{K}_2$.

\end{remark}

\vspace{0.3cm}

\section{Morse functions on the Jacobian}\label{chaincompl}
Suppose now that we fixed a spinorial perturbation $A$ small and generic so that the conclusions of Propositions \ref{noirred} and \ref{generic} both hold. Let us also make the following.
\begin{ass} The hypersurface $\mathsf{K}\subset\mathbb{T}_Y$ is smooth, i.e. $\mathsf{K}=\mathsf{K}_1$.
\end{ass}
This is indeed the case when $b_1\leq 3$ for dimensional reasons, and will also hold in the explicit examples with $b_1>3$ that we will consider later. By Lemma \ref{ktop}, the spectral flow between any two points in $\mathbb{T}_Y\setminus\mathsf{K}$ is well-defined, and we make the following definition.

\begin{defn}
We say that the family of Dirac operators on has type $k$ if the spectral flow between any two points in $\mathbb{T}_Y\setminus \mathsf{K}$ is bounded above in absolute value by $k$, and there are two points with spectral flow exactly $k$. We call a manifold $(Y,\spin)$ of the \textit{simplest type} if it has type $1$.
\end{defn}

When our manifold has type $k$, we can fix a basepoint $B_0$ so that all other points have positive spectral flow from it; this gives us a inclusion of subspaces
\begin{equation}\label{Tfilt}
\mathbb{T}_0\subset\mathbb{T}_1\subset\cdots\subset\mathbb{T}_k=\mathbb{T}_Y
\end{equation}
so that $\mathsf{K}=\amalg \partial\mathbb{T}_i$ and the points of $\mathrm{int}\mathbb{T}_j\setminus \mathbb{T}_{j-1}$ have spectral flow from $B_0$ exactly $j$. 
\begin{defn}\label{transversemorse}
We say that a Morse function $f:\mathbb{T}\rightarrow \mathbb{R}$ is $A$-\textit{transverse} if for each point $x\in\partial\mathbb{T}_j$, the vector field $-\mathrm{grad}(f)(x)$ is transverse to $\partial\mathbb{T}_j$ and points towards the interior of $\mathbb{T}_j$.
\end{defn}
Notice that $A$-transverse Morse functions always exist by Lemma \ref{ktop} and its proof.
\\
\par
The fact that $\mathsf{K}=\mathsf{K}_1$ is transversely cut out in the space of operators readily implies the following.
\begin{lemma}\label{smalleig}
Fix an $A$-transverse Morse function $f$, and consider a smooth curve $\gamma:(-\epsilon,\epsilon)\rightarrow \mathbb{T}_Y$ with the following properties:
\begin{itemize}
\item $\gamma$ intersects $\mathsf{K}$ transversely at $t=0$ at a point $\gamma(0)$;
\item $f(\gamma(t))> f(\gamma(0))$ (resp $<f(\gamma(0))$) for $t<0$ (resp $t>0$).
\end{itemize}
Then, after possibly restricting the interval, the operator $D_{\gamma(t)}-A$ has a well defined eigenvalue $\lambda(t)$ of smallest absolute value. Furthermore this eigenvalue is simple, and the function $t\mapsto\lambda(t)$ is differentiable and $\lambda'(0)<0$.
\end{lemma}

We consider then an additional perturbation of the form
\begin{equation}\label{morsepert}
\widetilde{\mathcal{L}}(B,\Psi)=\mathcal{L}(B,\Psi)-\frac{1}{2}\langle\Psi, A \Psi\rangle_{L^2}+\varepsilon f(B)
\end{equation}
for our $A$-transverse Morse function $f:\mathbb{T}\rightarrow\mathbb{R}$.
\begin{prop}
Suppose we are in the situation above. Then for $\varepsilon>0$ sufficiently small, the perturbed Seiberg-Witten equations corresponding to the critical points of (\ref{morsepert}), namely
\begin{align}
\begin{split}
D_{B}\Psi&=A\Psi\\
\frac{1}{2}\rho(\ast F_{B^t})+(\Psi\Psi^*)_0&=-\varepsilon\cdot \rho\left(\mathrm{grad}f(B)\right).
\end{split}
\end{align}
do not have irreducible solutions.
\end{prop}
Here, $\mathrm{grad}f(B)$ is a tangent vector on $\mathbb{T}_Y$, which we interpret as an imaginary-valued harmonic $1$-form on $Y$.
\begin{proof}
Let us denote for simplicity $D_{B,A}:=D_B-A$ throughout the proof. Suppose that there is a sequence $\varepsilon_i>0$ converging to zero such that the equation have an irreducible solution $(B_i,\Psi_i)$; these solve the equations in the blow-up $(B_i,r_i,\psi_i)$ given by
\begin{align}\label{pertSWblow}
\begin{split}
D_{B_i,A}\psi_i=0&\\
\frac{1}{2}\rho(\ast F_{B_i^t})+r_i^2(\psi\psi^*)_0&=-\varepsilon_i\cdot \rho\left( \mathrm{grad}f(B_i)\right)
\end{split}
\end{align}
where $\|\psi_i\|_{L^2}=1$. We can then extract a sequence converging to a solution $(B,r,\psi)$ with $\|\psi\|_{L^2}=1$ to the equations with $\varepsilon=0$; by Proposition \ref{noirred} we necessarily have $r=0$. In particular, $B^t$ is flat hence $B\in\mathbb{T}$, and the equation
\begin{equation*}
D_{B,A}\psi=0
\end{equation*}
then implies that $B\in\mathsf{K}$.
\par
Consider now a short piece of flowline $B(t)$ of $-\mathrm{grad}f$ in $\mathbb{T}$ such that $\gamma(0)=B$. There is a smooth parametrization $\lambda(t)$ of the smallest eigenvalue by Lemma \ref{smalleig}, and because the eigenvalue is simple we can choose a $1$-parameter family of non-zero eigenspinors $\varphi(t)$:
\begin{equation*}
D_{B(t),A}\varphi(t)=\lambda(t)\varphi(t)
\end{equation*}
with $\varphi(0)=\psi$.
\par
Differentiating we obtain
\begin{equation*}
D_{B(t),A}\dot{\varphi}(t)+\rho(\dot{B}(t))\varphi(t)=\dot{\lambda}(t)\varphi(t)+\lambda(t)\dot{\varphi}(t)
\end{equation*}
hence evaluating at $t=0$ we have
\begin{equation*}
D_{B,A}\dot{\varphi}(0)+\rho(-\mathrm{grad}f(B))\psi=\dot{\lambda}(0)\psi.
\end{equation*}
Taking $L^2$ inner products with $\psi$, and using that $D_{B,A}$ is self-adjoint, so that
\begin{equation*}
\langle D_{B,A}\dot{\varphi}(0),\psi\rangle_{L^2}=\langle \dot{\varphi}(0),D_{B,A}\psi\rangle_{L^2}=0
\end{equation*}
we obtain that
\begin{equation*}
\langle \rho\left(\mathrm{grad}f(B)\right), (\psi\psi^*)_0\rangle_{L^2}=\frac{1}{2}\langle\psi ,\rho(\mathrm{grad}f(B))\psi\rangle_{L^2}=-\frac{\dot{\lambda}(0)}{2}\|\psi\|_{L^2}^2>0.
\end{equation*}
where we used Lemma \ref{smalleig} for the sign on $\dot{\lambda}(0)$.
\par
On the other hand, we can take the $L^2$ inner product of the second equation in (\ref{pertSWblow}) with $\varepsilon_i\cdot \mathrm{grad}f(B_i)$; because $\ast F_{B^t}$ is exact (hence orthogonal to harmonic forms), we obtain that
\begin{equation*}
r_i^2\langle \rho\left(\mathrm{grad}f(B_i)\right), (\psi_i\psi_i^*)_0\rangle_{L^2}=-\varepsilon_i\|\mathrm{grad}f(B_i)\|^2_{L^2}\leq 0.
\end{equation*}
Because $r_i\neq 0$, we can conclude passing to the limit that
\begin{equation*}
\langle \rho\left(\mathrm{grad}f(B)\right), (\psi\psi^*)_0\rangle_{L^2}\leq 0,
\end{equation*}
which is a contradiction.
\end{proof}

With this in hand, we can furthermore add a small additional perturbation as in \cite{KM} to make first the spectra of the perturbed Dirac operators simple at the critical points of $f$, and then to arrange regularity for the moduli spaces of flowlines, hence achieving transversality without introducing irreducible critical points.
\par
The Floer chain complex computing $\HMt_*(Y,\spin)$ is then is generated by the stable critical points in the blow-up $\check{C}_*=C^s$; these form an infinite tower over each critical point of $f$ in $\mathbb{T}$. The differential is given by 
\begin{equation*}
\check{\partial}=\bar{\partial}^s_s-\partial^u_s\bar{\partial}^s_u.
\end{equation*}
Unfortunately, in this level of generality, the latter cannot be described concretely for two reasons:
\begin{itemize}
\item the term $\partial^u_s$ involves irreducible solutions to the Seiberg-Witten equations on $\mathbb{R}\times Y$, of which we do not have any understanding.
\item the counts of reducible trajectories $\bar{\partial}^s_s$ and $\bar{\partial}^s_u$ will in general involve families of Dirac operators parametrized by higher dimensional spaces of Morse flowlines.
\end{itemize} 
We conclude this section by describing two useful tools that allow to obtain partial information about $\HMt_*(Y,\spin)$; we will see in the next section that when combined, these provide a complete description of the homology group in the simplest type case.
\par
First of all, one can compare the Floer chain complex $\check{C}_*$ with the Floer chain complex computing the $U$-localized version $\HMb_*(Y,\spin)$. This is given by $\bar{C}_*=C^s\oplus C^u$ with differential
\begin{equation*}\bar{\partial}=
\begin{bmatrix}
\bar{\partial}^s_s & \bar{\partial}^u_s\\
\bar{\partial}^s_u & \bar{\partial}^u_u\\
\end{bmatrix}.
\end{equation*}
The map $i_*:\HMb_*(Y,\spin)\rightarrow \HMt_*(Y,\spin)$, which is induced by the chain map
\begin{equation*}
i:C^s\oplus C^u\rightarrow C^s
\end{equation*}
given by $i=\mathrm{id}-\partial^u_s$ induces an isomorphism in degrees high enough. This is helpful because even though one cannot descrive explicitly the differential $\bar{\partial}$, the resulting homology groups $\HMb_*(Y,\spin)$ can be readily computed in terms of the triple cup product $\cup_Y^3$, as we briefly recall.
\par
In \cite[Ch. 33]{KM}, the authors define for a closed smooth manifold $Q$ equipped with a family $L$ of Dirac type operators the coupled Morse homology group $\bar{H}_*(Q,L)$. This is an invariant of the homotopy type of $L$, and is obtained by choosing a Morse function $f:Q\rightarrow \mathbb{R}$ and studying the negative gradient flow equation for $f$ suitably coupled with the Dirac equation. When $Q=\mathbb{T}_Y$ and $L$ is the corresponding family of Dirac operators, the outcome is exactly $\HMb_*(Y,\spin)$. Using the extra flexibility afforded by homotopy invariance, the authors were able to provide the following general computation of $\HMb_*(Y,\spin)$ over $\mathbb{Q}$. Consider $H_*(\mathbb{T}_Y)\otimes \mathbb{Q}[U,U^{-1}]$. Identifying $H_k(\mathbb{T}_Y)=\Lambda^k H_1(\mathbb{T}_Y)=\Lambda^k H^1(Y)$, we can consider the degree $-1$ map
\begin{equation*}
d:=\iota_{\cup_Y^3}\otimes U^{-1},
\end{equation*}
where $\iota_{\cup_Y^3}$ is the contraction with $\cup_Y^3\in\left(\Lambda^k H^1(Y)\right)^*$. Then $\HMb_*(Y,\spin)=\mathrm{ker}d/\mathrm{im}d$.
\par
The second tool arises from the filtration of spaces (\ref{Tfilt}). Let us introduce the following refinement of the notion of $A$-transverse Morse function.
\begin{defn}\label{adaptedmorse}
We say that a Morse function $f:\mathbb{T}_Y\rightarrow \mathbb{R}$ is $A$-\textit{adapted} if:
\begin{itemize}
\item each $j=0,\dots, k$ is a regular value of $f$ and $f^{-1}(j)=\partial\mathbb{T}_j$;
\item for each $j$, $f(\mathbb{T}_j\setminus \mathrm{int}\mathbb{T}_{j-1})=[j-1,j]$.
\end{itemize}
In particular, $f(\mathbb{T}_Y)=[-1,k]$.
\end{defn}
Again, $A$-adapted Morse functions exist by Lemma \ref{ktop} and its proof. The meaning of this condition is the following. The value of the perturbed functional (\ref{morsepert}) at a reducible critical point $(B,0)$ is given by $f(B)$. Because of this, for a given choice of $A$-adapted $f$, we then see that the differential of a critical point lying in $\mathbb{T}_j$ still lies in $\mathbb{T}_j$, hence we obtain a natural filtration of $\check{C}_*$.
\begin{remark}
Relatedly, it is important to remark that the Morse index of the underlying critical point of $f$ does \textit{not} define a filtration on $\check{C}_*$ because of the term $\partial^u_s$ involving irreducible trajectories. On the other hand, this is true in specific circumstances, one of which will be instrumental in the proof of Theorem \ref{main} in Section \ref{T3}.
\end{remark}
Our discussion in the next section (which we will spell out in detail in the case of simplest type, in which more can be said) will imply the following result (the non-obvious part being the identification of the $E^1$-page). 
\begin{prop}\label{sfss}
In our setup, there exists a spectral sequence whose $E^1$ page is given by
\begin{equation*}
E^1=\bigoplus H_*(\mathbb{T}_j,\mathbb{T}_{j-1})\langle -2j\rangle\otimes \mathcal{T}_+
\end{equation*}
and which converges to $\HMt_*(Y,\spin)$.
\end{prop}
Here $\mathcal{T}_+$ is the $\mathbb{Q}[U]$-module $\mathbb{Q}[U^{-1},U]/U\cdot\mathbb{Q}[U]$, commonly referred to as `tower'. Let us point out that in many situations it is very important to understand the $U$-action on $\HMt_*(Y,\spin)$; while the spectral sequence is compatible with the $U$-action, it is in general non-trivial to reconstruct the $U$-action on $\HMt_*(Y,\spin)$ from that on the $E^{\infty}$-page.

\vspace{0.3cm}

\section{Spin$^c$ Manifolds of the simplest type}\label{simplestdef}

We focus on the case of a spin$^c$ manifold $(Y,\spin)$ of the simplest type (where we keep the metric $g$ and spinorial perturbation $A$ implicit in our discussion). In this situation, we denote the filtration (\ref{Tfilt}) simply by
\begin{equation*}
\mathbb{T}_-\subset\mathbb{T}.
\end{equation*}
By Lemma \ref{ktop}, the restriction of the family of Dirac operators $\{D_B-A\}$ is homotopically trivial on both $\mathbb{T}_-$ and $\mathbb{T}\setminus\mathrm{int}\mathbb{T}_-$. On the other hand, the family on $\mathbb{T}$ itself is not homotopically trivial in general, the obstruction being the triple cup product of the three-manifold (\ref{cup3}).
\begin{remark}
We can interpret this in terms of the cofibration
\begin{equation*}
\mathbb{T}_-\stackrel{\iota}{\hookrightarrow} \mathbb{T}\rightarrow \mathbb{T}/\mathbb{T}_-
\end{equation*}
as follows. Families of Dirac type are classified up to homotopy by maps to $U(\mathbb{\infty})$; using $[X,U(\infty)]=K^1(X)$, we can consider then the morphism of exact sequences
\begin{center}
\begin{tikzcd}
	{K^1(\mathbb{T}/\mathbb{T}_-)} & {K^1(\mathbb{T})} & {K^1(\mathbb{T}_-)} \\
	{H^{\mathrm{odd}}(\mathbb{T}/\mathbb{T}_-;\mathbb{Q}))} & {H^{\mathrm{odd}}(\mathbb{T};\mathbb{Q})} & {H^{\mathrm{odd}}(\mathbb{T}_-,;\mathbb{Q}))}
	\arrow[from=1-1, to=1-2]
	\arrow["\iota^*" ,from=1-2, to=1-3]
	\arrow[from=2-1, to=2-2]
	\arrow["\iota^*" ,from=2-2, to=2-3]
	\arrow[from=1-1, to=2-1]
	\arrow[from=1-2, to=2-2]
	\arrow[from=1-3, to=2-3]
\end{tikzcd}
\end{center}
where the vertical maps are given by the Chern character. In particular we know from the Atiyah-Singer index theorem that
\begin{equation*}
\mathrm{ch}(\mathrm{ind}\{D_B-A\})=\cup_Y^3\in H^3(\mathbb{T})=\left(\Lambda^1H^1(Y)\right)^*
\end{equation*}
so that $i^*(\cup_Y^3)=0\in H^3(\mathbb{T}_-)$, hence $\cup_Y^3$ arises from an element of $H^3(\mathbb{T},\mathbb{T}_-;\mathbb{Q})$.
\end{remark}

We then choose an adapted Morse function $f$, see Definition \ref{adaptedmorse}. The Floer chain complex $\check{C}$ is naturally relatively $\mathbb{Z}$-graded in this situation. For our purposes it is convenient to fix the absolute grading of the first stable critical point over an index zero critical point in $\mathbb{T}_-$ to be $0$; this is well-defined because there is no spectral flow between points in $\mathbb{T}_-$ of $f$ given our definition.
\begin{remark}
Notice that this is not the canonical absolute $\mathbb{Q}$-grading introduced in \cite[Ch. 28]{KM}.
\end{remark}

Let us consider the filtration on critical points induced by the value of $\widetilde{\mathcal{L}}$. We obtain subcomplexes
\begin{equation*}
\check{C}^-\subset \check{C}\text{  and  }\bar{C}^-\subset \bar{C}.
\end{equation*}
Denoting by $\check{C}^+$ and $\bar{C}^+$ the respective quotient chain complexes, these all fit in the diagram
\begin{center}
\begin{tikzcd}
	0 & {\bar{C}^-} & {\bar{C}} & {\bar{C}^+} & 0 \\
	0 & {\check{C}^-} & {\check{C}} & {\check{C}^+} & 0
	\arrow[from=2-1, to=2-2]
	\arrow[from=2-2, to=2-3]
	\arrow[from=2-3, to=2-4]
	\arrow[from=2-4, to=2-5]
	\arrow[from=1-1, to=1-2]
	\arrow[from=1-2, to=1-3]
	\arrow[from=1-3, to=1-4]
	\arrow[from=1-4, to=1-5]
	\arrow[from=1-2, to=2-2]
	\arrow[from=1-3, to=2-3]
	\arrow[from=1-4, to=2-4]
\end{tikzcd}
\end{center}
Calling the vertical maps $i$, we have the following observation.
\begin{lemma}\label{surj}
The induced maps $i_*:H_*(\bar{C}^{\pm})\rightarrow H_*(\check{C}^{\pm})$ are surjective, and isomorphisms in degrees high enough.
\end{lemma}
Notice that the analogous statement does not hold for the map $i_*$ on the full complexes; indeed, in that situation the cokernel is the reduced Floer homology, which we will see to be not-trivial in concrete examples.
\begin{proof}
The key point here is that the restrictions of the family of operators to the complement of $\mathsf{K}\subset\mathbb{T}$ have no kernel, hence there is no spectral flow along paths contained in either $\mathrm{int}\mathbb{T}_-$ or $\mathbb{T}\setminus\mathrm{int}\mathbb{T}_-$. More specifically, let us consider $\mathbb{T}_-$, the other case being the same.
\par
We first show that $\bar{\partial}^s_u=0$. The $k$th stable critical point over an index $i$ critical point (where the bottom corresponds to $k=0$) of $f$ has grading $i+2k$. On the other hand, the $j$th unstable critical point (where the top corresponds to $j=0$) has grading $i-1-2j$. Because trajectories decrease Morse indices by at least $1$, we conclude for grading reasons.
\par
Inspecting the definition of $\bar{\partial}$, we see that $\bar{\partial}^s_u=0$ implies that if $x\in\check{C}^-=C^s$ is a cycle, then $(x,0)\in\bar{C}^-$ is also a cycle, so that $i_*$ is surjective. Finally, $i_*$ is an isomorphism in degrees high enough because there is an upper bound on the degree of unstable critical points.
\end{proof}

We now discuss the groups $H_*(\bar{C}^{\pm})$ in more detail. In order to do so, we generalize the framework of coupled Morse homology \cite[Ch. 33]{KM} from the setup of closed manifold manifolds $Q$ to the case of non-empty boundary $\partial Q\neq 0$. In this situation, as in classical Morse theory with boundary, one can look at Morse functions $f$ for which $-\mathrm{grad}f$ points either outwards or inwards at the boundary; in the former case one obtains $H_*(Q,\partial Q)$ while in the latter one obtains $H_*(Q)$. When $Q$ is equipped with a family $L$ of Dirac-type operators, one can then construct the corresponding coupled Morse homology groups by $\bar{H}^{\pm}_*(Q;L)$. These are again invariant under homotopy of $L$.
\par
In our situation, we apply this discussion to $\mathbb{T}_-$ and $\mathbb{T}\setminus\mathrm{int}\mathbb{T}_-$ with a Morse function on $\mathbb{T}$ as above; thought of a Morse function on $\mathbb{T}\setminus\mathrm{int}\mathbb{T}_-$ (resp $\mathbb{T}^-$), the vector field $-\mathrm{grad}f$ points outwards (resp. inwards), and the respective coupled Morse chain complexes are identified with $\bar{C}^+$ (resp. $\bar{C}^-$). Hence we get the exact triangle
\begin{center}
\begin{tikzcd}
	&& {\HMb_*(Y,\spin)} \\
	{\bar{H}^-_*(\mathbb{T}^-,\{D_{B,A}\})} &&&& {\bar{H}^+_*(\mathbb{T}\setminus\mathrm{int}\mathbb{T}_-,\{D_{B,A}\})}
	\arrow[from=1-3, to=2-5]
	\arrow["{\bar{\partial}}",from=2-5, to=2-1]
	\arrow[from=2-1, to=1-3]
\end{tikzcd}
\end{center}
where the horizontal arrow has degree $-1$.
\par
Because the restriction of the family $\{D_B-A\}$ to $\mathbb{T}\setminus\mathsf{K}$ is homotopically trivial, we also have the further identifications
\begin{align}\label{collapse}
\begin{split}
{\bar{H}_*^+(\mathbb{T}\setminus\mathrm{int}\mathbb{T}_-,\{D_{B,A}\})}&\cong H_*(\mathbb{T},\mathbb{T}_-)\otimes \mathbb{Q}[U,U^{-1}]\\
{\bar{H}_*^-(\mathbb{T}^-,\{D_{B,A}\})}&\cong H_*(\mathbb{T}^-)\otimes \mathbb{Q}[U,U^{-1}].
\end{split}
\end{align}
where we identified the homology of $\mathbb{T}\setminus\mathrm{int}\mathbb{T}_-$ relative to its boundary with the homology of the pair $(\mathbb{T},\mathbb{T}_-)$ using excision.
\\
\par
Using this, we can obtain a description of the groups $H_*(\check{C}^{\pm})$ as follows.
\begin{lemma}We have the identifications
\begin{align*}
H_*(\check{C}^-)&\cong H_*(\mathbb{T}_-)\otimes\mathcal{T}_+\\
H_*(\check{C}^+)&\cong H_*(\mathbb{T},\mathbb{T}_-)\otimes\mathcal{T}_+\langle-2\rangle
\end{align*}
as graded $\mathbb{Q}[U]$-modules (up to an overall shift).
\end{lemma}

\begin{proof}
Because $\bar{\partial}^s_u=0$ on $\check{C}^{\pm}$ by the proof of Lemma \ref{surj}, we see that the differential preserves the filtration by Morse index. 
\par
Let us choose an adapted Morse function with the following (readily arranged via handleslides) additional property: when restricted to $\mathbb{T}_-$ (resp. $\mathbb{T}\setminus\mathbb{T}_-$), for all critical points $x,y$ we have $\mathrm{ind}(y)<\mathrm{ind}(x)$ if and only if $f(y)<f(x)$. Because of this, the map $\partial^u_s$ preserves the Morse index filtration when restricted to $\mathbb{T}_-$ and $\mathbb{T}\setminus\mathbb{T}_-$; hence the maps
\begin{equation*}
i:\bar{C}^{\pm}\rightarrow\check{C}^{\pm}
\end{equation*}
are maps of filtered chain complexes. By (\ref{collapse}), the corresponding spectral sequence on $\bar{C}^{\pm}$ collapses at the $E^2$ page. Our claimed computation is the $E^2$-page of the spectral sequence on $\check{C}^{\pm}$; we conclude because the spectral sequence collapses at this stage by surjectivity of $i_*$ on $E^2$, see Lemma \ref{surj}.
\par
Finally, the grading shift in $H_*(\check{C}^+)$ arising from the spectral flow when crossing $\mathsf{K}$.
\end{proof}
In particular, we have an exact triangle
\begin{center}
\begin{tikzcd}\label{triangle1}
	&& {\HMt_*(Y,\spin)} \\
	{ H_*(\mathbb{T}_-)\otimes\mathcal{T}_+} &&&& {H_*(\mathbb{T},\mathbb{T}_-)\otimes\mathcal{T}_+\langle-2\rangle}
	\arrow[from=1-3, to=2-5]
	\arrow["{\check{\partial}}",from=2-5, to=2-1]
	\arrow[from=2-1, to=1-3]
\end{tikzcd}
\end{center}
where again the horizontal arrow has degree $-1$. Furthermore, the natural maps $i_*$ from the bar version triangle are
all surjective. Because of this, in order to describe $\check{\partial}$ it will suffice to describe $\bar{\partial}$, which we now do.
\\
\par
Consider the long exact sequence of the pair $(\mathbb{T},\mathbb{T}_-)$ in usual homology with $\mathbb{Q}$-coefficients
\begin{center}
\begin{tikzcd}
	&& {H_*(\mathbb{T})} \\
	{H_*(\mathbb{T}_-)} &&&& {{H}_*(\mathbb{T},\mathbb{T}_-)}
	\arrow["{q_*}",from=1-3, to=2-5]
	\arrow["{\partial}",from=2-5, to=2-1]
	\arrow["{i_*}",from=2-1, to=1-3]
\end{tikzcd}
\end{center}
where $\partial$ has degree $-1$. Because we are using $\mathbb{Q}$-coefficients, this breaks into the short exact sequences
\begin{equation}\label{short}
	0 \rightarrow {i_*(H_k(\mathbb{T}_-))} \rightarrow {H_k(\mathbb{T})} \rightarrow {q_*(H_k(\mathbb{T}))} \rightarrow 0
\end{equation}
and the degree $-1$ isomorphism 
\begin{equation}\label{isoconnecting}
 \mathrm{coker}\left(q_*:H_k(\mathbb{T})\rightarrow{H}_k(\mathbb{T},\mathbb{T}_-)\right)\cong\mathrm{ker}\left(i_*:H_{k-1}(\mathbb{T}_-)\rightarrow H_{k-1}(\mathbb{T})\right)
\end{equation}
Denoting the pieces in (\ref{short}) by $I^\pm_*$ respectivelty, and (\ref{isoconnecting}) by $E_*$, we see can consider the map $\bar{\partial}$ under the identification (\ref{collapse}) as a map
\begin{equation*}
\bar{\partial}:(I^+_*\oplus E_*)\otimes \mathbb{Q}[U,U^{-1}]\rightarrow (I^-_*\oplus E_*\langle-1\rangle)\otimes\mathbb{Q}[U,U^{-1}].
\end{equation*}
As we are working in the localized (bar) version of the invariants, this map is in fact independent of the specific family of operators up to homotopy and choice of Morse function. In order to obtain a concrete description, we look at the determination of the chain complex for a family pulled-back from $SU(2)$ as discussed in \cite[Ch. 34]{KM}, and in particular Proposition $34.2.1$ and its proof.
\par
Referring to the construction and notation there, because in our setup the restriction of the family on $\mathbb{T}\setminus\mathrm{int}\mathbb{T}_-$ is homotopically trivial, we can begin by homotoping the whole family to a map $v:\mathbb{T}\rightarrow SU(2)$ so that the restriction to it is the constant operator taking values in the north pole of $SU(2)$. In particular, $v^{-1}(W)\subset \mathbb{T}_-$, and we can furthermore choose the Morse function to take a fixed regular value on $\mathsf{K}$.
\par
With this setup, the coupled Morse chain complex is then $C_*(\mathbb{T},f)\otimes \mathbb{Q}[T,T^{-1}]$ equipped with the differential
\begin{align*}
\bar{\partial}x=\partial x+ (\tilde{\xi}\cap x)T
\end{align*}
where $\tilde{\xi}$ is a \v Cech cochain representing $\cup^3_Y$ supported in $\mathbb{T}_-$. When considering the filtration induced by $\mathbb{T}_-\subset\mathbb{T}$, we see then that $\bar{\partial}x=\partial x$ on $C_*(\mathbb{T},\mathbb{T}_-)\otimes \mathbb{Q}[T,T^{-1}]$. In particular we conclude (omitting the powers of $T$ from the notation for simplicity) that:
\begin{itemize}
\item on $I_+$, $\bar{\partial}$ coincides with cap product with $[\tilde{\xi}]=\cup_Y^3$;
\item on $E$, $\bar{\partial}$ is given by the isomorphism (\ref{isoconnecting}) (which decreases the Morse filtration index by $1$) plus higher order terms (possibly involving $I_-$); in particular it is an isomorphism onto its image.
\end{itemize}
Putting everything together, we finally get the following.
\begin{thm}\label{simplestcomp}
Suppose $(Y,\spin)$ is of the simplest type. We have the decomposition of $\mathbb{Q}[U]$-modules
\begin{equation*}
\HMt_*(Y,\spin)=i_*\left(\HMb_*(Y,\spin)\right)\oplus \HMr_*(Y,\spin)
\end{equation*}
where
\begin{align*}
i_*\left(\HMb_*(Y,\spin)\right)&\cong \mathrm{Cone}\left(\cup_Y^3:I^+\langle-2\rangle\rightarrow I^-\right)\otimes \mathcal{T}^+\\
\HMr_*(Y,\spin)&\cong E_*\langle-2\rangle;
\end{align*}
this holds at the level of graded modules up to an overall grading shift.
\end{thm}
\begin{proof}
All the pieces are in places so we will comment on the remaining details. The grading shifts arise from the spectral flow when crossing $\mathsf{K}$, and the reduced Floer homology arise as the kernel of $\check{\partial}$ from parts of the triangle of the form
\begin{equation}\label{dashed}
U:\mathcal{T}^+\langle-2\rangle\rightarrow  \mathcal{T}^+\langle-1\rangle,
\end{equation}
where the $-1$ shift is the one from (\ref{isoconnecting}). Finally, there are no extension problems to worry about because the reduced homology on the $E^{\infty}$ page lies in the top filtration level, and $U$-action from towers can only decrease it.
\end{proof}

\vspace{0.3cm}

\section{Concrete examples}\label{sigmag}

In this section we study two concrete examples of spin$^c$ manifolds of simplest type (for suitable metric $g$ and perturbation $A$), namely $S^1\times \Sigma_h$ with $h=2,3$ equipped with the unique torsion spin$^c$ structure $\spin_0$, and use Theorem \ref{simplestcomp} to provide a computation of their Floer homology. We will suppress the genus of the surface from our notation when unneccessary.
\par
It is well-known that the unperturbed equations on $(S^1\times \Sigma,\spin_0)$ equipped with a product metric do not admit irreducible solutions for all $h$, see for example \cite{LinTheta} for an exposition. Furthermore the torus $\mathbb{T}_{S^1\times \Sigma}$ of flat connections can be identified with $S^1\times \mathrm{Jac}(\Sigma)$. The corresponding Dirac operator has kernel exactly at points in $\{0\}\times \Theta$, where $\Theta$ is the classical theta divisor of the complex curve: after identifying $\mathrm{Jac}(\Sigma)$ with the torus of degree $h-1$ holomorphic line bundles $\Theta\subset\mathrm{Jac}(\Sigma)$ is the locus consisting of bundles admitting non-zero holomorphic sections. Notice that this depends on the metric on $\Sigma$ only through its underlying conformal structure.
\par
Unfortunately, the argument for the absence of irreducible solutions (which involves dimensional reduction) is quite subtle and does not apply to the perturbed equations. We will therefore take an alternative route involving spectral geometry. In particular, we have the following.

\begin{prop}
Consider the product metric on $S^1\times\Sigma$ for which the length of the $S^1$ factor is small enough. Then $\lambda_1^*=\lambda_1(\Sigma)$, where the latter is the first non-zero eigenvalue of the Laplacian on functions on $\Sigma$.
\end{prop}

\begin{proof}
The result follows from a direct computation which is well-known in classical electromagnetism in relation to the propagation of the electromagnetic field in a waveguide \cite{Wal}. We denote by $\ast_\Sigma$ and $\ast$ the Hodge star on $\Sigma$ and $S^1\times\Sigma$ respectively. For a $1$-form $a$ on $\Sigma$, we have that
\begin{align*}
\ast a&=\ast_\Sigma a\wedge dt\\
\ast(a\wedge dt)&=-\ast_\Sigma a.
\end{align*}
Consider a general $1$-form
\begin{equation*}
b=fdt+\alpha
\end{equation*}
where $f$ is a function on $Y$ and $\alpha$ is a $1$-parameter family of $1$-forms on $\Sigma$.
Let us consider the operator $\ast d$ (which squares to $\Delta$ on coclosed $1$-forms)
\begin{align*}
\ast db&=\ast(d_{\Sigma}fdt+d_\Sigma\alpha-\alpha'dt)=\\
&=-\ast_\Sigma d_\Sigma f+\ast_\Sigma \alpha'+(\ast_{\Sigma}d_\Sigma \alpha)dt.
\end{align*}
Hence the eigenvalue equation $\ast db=\mu b$ takes the form
\begin{align*}
\ast_{\Sigma}d_\Sigma \alpha&=\mu f\\
-\ast_\Sigma d_\Sigma f+\ast_\Sigma \alpha'&=\mu\alpha.
\end{align*}
Notice that solutions to the eigenvalue equation are automatically coclosed, and correspond to eigenforms with eigenvalue $\mu^2$. Let us assume $\mu\neq 0$ (harmonic forms are spanned by $dt$ and the harmonic forms on $\Sigma$). Substituting the first equation in the second one we then get (as $d^*_\Sigma=-\ast_\Sigma d_\Sigma\ast_\Sigma$ on $2$-forms on a surface)
\begin{equation*}
d^*_\Sigma d_\Sigma {\alpha}+\mu\ast_\Sigma \alpha'=\mu^2\alpha.
\end{equation*}
We now consider a basis of $1$-eigenforms $e_n$, $h_k$ and $c_n$ (exact, harmonic and coexact respectively), where we can assume $c_n=\ast_\Sigma e_n$ have eigenvalue $\lambda_n\neq 0$ for $\Delta_\Sigma=d_\Sigma d^*_\Sigma+d^*_\Sigma d_\Sigma$, and $\ast h_i=h_{i+g}$ for $i=1,\dots, g$. If $S^1=\mathbb{R}/2\pi L \mathbb{Z}$ has coordinate $t$, expanding in Fourier series in the $t$ variable
\begin{equation*}
\alpha(t)=\sum E_n(t)e_n+\sum C_n(t)c_n+\sum_{i\leq g} H_i(t)h_i+ H_{i+g}(t)h_{i+g},
\end{equation*}
with $E_n, C_n$ and $H_k$ periodic of period $2\pi L$, we have then
\begin{align*}
d^*_\Sigma d_\Sigma\alpha(t)&=\sum \lambda_n\cdot C_n(t)c_n\\
\ast_\Sigma \alpha'&=-\sum C_n'(t)e_n+\sum E_n'(t)c_n+\sum_{i\leq g}-H_{i+g}'(t)h_{i}+H_i'(t)h_{i+g}.
\end{align*}
Hence, looking at the $c_n$ and $e_n$ component we get the system
\begin{align*}
\lambda_n C_n(t)+\mu E'_n(t)&=\mu^2C_n(t)\\
-\mu C_n'(t)&=\mu^2 E_n(t).
\end{align*}
and looking at the $h_i$ components we obtain the system
\begin{align*}
H_i'&=\mu H_{i+g}\\
H_{i+g}'&=-\mu H_{i},
\end{align*}
which we can address separately. In the first system, differentiating the second equation, and substituting in the first, we obtain
\begin{equation*}
C_n''(t)=(\lambda_n-\mu^2)C_n(t).
\end{equation*}
In order for this equation to have periodic solutions (as $C_n$ has to be) we necessarily need $\lambda_n\leq \mu^2$. Conversely, one readily sees that $c_1$ defines a coexact eigenform on $S^1\times \Sigma$ with eigenvalue $\lambda_1(\Sigma)$.
\par
The second system involves the equation
\begin{equation*}
H_k''(t)=-\mu^2 H_k(t).
\end{equation*}
On $S^1=\mathbb{R}/2\pi L \mathbb{Z}$, the Fourier eigenmodes are $\{e^{int/L}\}$ with respective eigenvalue $n^2/L^2$. Hence by choosing $L$ sufficiently small these corresponding coexact eigenforms will have eigenvalue strictly larger than $\lambda_1(\Sigma)$.
\end{proof}

Suppose now that $\Sigma$ is equipped with a hyperbolic metric (by uniformization, there is a unique one in any conformal class). Then the Ricci tensor of $S^1\times \Sigma$ for the product metric is
\begin{equation*}
\begin{bmatrix}
0&0&0\\
0&-1&0\\
0&0&-1
\end{bmatrix}
\end{equation*}
so that $\tilde{s}\equiv -2$. By Proposition \ref{noirred} we can therefore conclude the following.
\begin{cor}
Suppose $(\Sigma,g_{hyp})$ is a hyperbolic surface for which $\lambda_1>1$. Then for the product metric on $S^1\times \Sigma$, where the $S^1$ component has sufficiently small length, small spinorial perturbations of the Seiberg-Witten equations do not have irreducible solutions.
\end{cor}
We will focus on the two following examples of low-genus hyperbolic surfaces for which $\lambda_1>1$:
\begin{itemize}
\item the Bolza surface (which corresponds to the affine equation $y^2=x^5-x$) has genus $2$ and $\lambda_1\approx 3.84$ \cite{SU}.
\item the Klein quartic (which corresponds to the homogeneous equation $x^3y+y^3z+z^3x=0$) has genus $3$ and $\lambda_1\approx 2.68$ \cite{Coo}.
\end{itemize}

To proceed further, we need to study the locus in $\mathbb{T}$ where the perturbed Dirac equation has solutions. Let us briefly recall the situation for the unperturbed Dirac equation, see \cite{LinTheta} for more details. Fix a square root $L$ of the canonical line bundle $K$ on $\Sigma$, so that $L\otimes \bar{K}\equiv \bar{L}$; this is equivalent to a choice of spin structure on $\Sigma$. We then have for $C$ a connection obtained from the spin connection one by twisting by a flat $U(1)$-connection on $\Sigma$ the operator
\begin{equation*}
\bar{\partial}_C:\Gamma(L)\rightarrow\Gamma(\bar{L})
\end{equation*}
and its adjoint $\bar{\partial}^*_C$. The spinor bundle on $S^1\times\Sigma$ (equipped with the product metric) is then $S=L\oplus \bar{L}$, and the Dirac operator corresponding to $B=C+i\lambda dt$ is
\begin{equation*}
D_B=
\begin{bmatrix}
i\frac{d}{dt}-\lambda& \sqrt{2}\bar{\partial}^*_C\\
\sqrt{2}\bar{\partial}_C&-i\frac{d}{dt}+\lambda
\end{bmatrix}
\end{equation*}
From here one concludes that if $D_B$ has non-trivial kernel, after gauge transformation $\lambda=0$ and solutions $(\alpha,\beta)$ satisfy $\bar{\partial}_C\alpha =0$ and $\bar{\partial}_C^*\beta=0$ (cf. \cite{LinTheta} for an exposition). In particular $B$ defines a holomorphic line bundle of degree $h-1$ with non-trivial holomorphic sections, i.e. it belongs to the theta divisor $\Theta\subset \mathrm{Jac}(\Sigma)$ as previously mentioned. The latter is a very well-studied object in the geometry of algebraic curves. It is the image of the Abel-Jacobi map
\begin{equation*}
\mathcal{A}:\mathrm{Sym}^{h-1}(\Sigma_{h-1})\rightarrow\mathrm{Jac}(\Sigma_h).
\end{equation*}
The map is generically injective, and the inverse image of a holomorphic line bundle is the projectivization of its space of holomorphic sections. In low genus, we have the following results (see for example \cite{ACGH}):
\begin{itemize}
\item When $h=2$, the map is always an isomorphism onto its image.
\item When $h=3$ is generic (not hyperelliptic), the map is again an isomorphism onto its image.
\end{itemize}
Notice that the Klein quartic is not-hyperelliptic, because smooth quartics are canonically embedded (see for example \cite[Ch. VII]{Mir}).
Let us recall some well-known facts about the topology of symmetric products \cite{Mac}:
\begin{itemize}
\item When $h=2$, the map $\mathcal{A}:\Sigma\rightarrow \mathrm{Jac}(\Sigma)$ induces an injection in homology.
\item When $h=3$, we have
\begin{equation}
H_i\left(\mathrm{Sym}^2(\Sigma_3)\right)=
\begin{cases}\label{sym3}
\mathbb{Q}\text{ if }i=0,4\\
H_1(\Sigma)\text{ if }i=1,3\\
\Lambda^2H_1(\Sigma)\oplus\mathbb{Q}\text{ if }i=2.
\end{cases}
\end{equation}
Furthermore, the induced map $\mathcal{A}_*$ has kernel given by the extra $\mathbb{Q}$ summand in $H_2$.
\end{itemize}
When perturbing the equations, we have the following.
\begin{lemma}\label{perttheta}
Consider a genus $2$ surface or a non-hyperelliptic genus $3$ surface. For $\delta>0$ small enough the locus $\mathsf{K}\subset\mathbb{T}$ for which the perturbed Dirac operator $\{D_B-\delta\}$ has kernel is transversely cut out, smooth and diffeomorphic to an $S^2$-bundle over $\Theta$. The component $\mathbb{T}_-$ is diffeomorphic to a disk bundle over $\Theta$, hence in particular $\mathbb{T}_-$ has the same homology as $\mathrm{Sym}^{h-1}(\Sigma_h)$.
\end{lemma}
\begin{remark}
The same is true in the case of genus $0$ and $1$ surfaces, in which case $\Theta$ are empty and a point respectively.
\end{remark}
\begin{proof}
The theta divisor is the zero set of the theta function, and the Riemann singularity theorem says that the multiplicity of the theta function at a point $L$ is given by $h^0(L)$ \cite[Ch. VI]{ACGH} by our discussion above, this is $1$ at all points under the assumption of the lemma. In our context, the Riemann singularity theorem and its proof mean that the map
\begin{align*}
\mathrm{Jac}(\Sigma)&\rightarrow\mathrm{Fred}_0\\
C&\mapsto \bar{\partial}_C,
\end{align*}
where $\mathrm{Fred}_0$ is the space of index zero Fredholm operators (in a suitable analytical setup), is transverse to the locus of operators with kernel. This is because $\bar{\partial}_C$, the normal bundle is identified with operators on $\mathrm{ker}(\bar{\partial}_C)$, i.e. a copy of $\mathbb{C}$.
\par
We can then reduce to study the linearization of the local model in the normal directions. In particular, after a quick analysis of the Fourier expansion as in \cite{LinTheta}, we are left to understand the locus is the space of matrices
\begin{equation*}
\begin{bmatrix}
-\lambda& \bar{a}\\
a& \lambda
\end{bmatrix}\text{ where }\lambda\in\mathbb{R},a\in\mathbb{C}
\end{equation*}
having eigenvalue $\delta$. This is exactly the locus $\lambda^2+|a|^2=\delta^2$ (where the multiplicity is always $1$), i.e. a copy of $S^2$.
\end{proof}

Putting everything together, we have proved the following.
\begin{cor}
Consider on $(S^1\times \Sigma_h,\spin_0)$ a product metric with a short length $S^1$ and either the Bolza surface ($h=2$) or the Klein quartic ($h=3$). Then for small $\delta\neq 0$, the perturbation is of the simplest type.
\end{cor}

Using Theorem \ref{simplestcomp} we then obtain the following concrete computations; the analogue in Heegaard Floer homology was proved using surgery techniques in \cite{OS}, \cite{JM}.
\begin{cor} We have the isomorphism of absolutely graded $\mathbb{Q}[U]$-modules
\begin{align*}
\HMt_*(S^1\times \Sigma_2,\spin_0)=& \left(\mathbb{Q}_{-1}\oplus \mathbb{Q}^{\oplus 9}_{-2}\oplus \mathbb{Q}^{\oplus 9}_{-3}\oplus\mathbb{Q}_{-4}\right)\otimes \mathcal{T}^+\\
\HMt_*(S^1\times \Sigma_3,\spin_0)=& \left(\mathbb{Q}_{-1}\oplus\mathbb{Q}^{\oplus 6}_{-2}\oplus\mathbb{Q}^{\oplus 28}_{-3}\oplus \mathbb{Q}^{\oplus 28}_{-4}\oplus\mathbb{Q}^{\oplus 6}_{-5} \mathbb{Q}_{-6} \right)\otimes \mathcal{T}^+\\
&\oplus \HMr_*(S^1\times \Sigma_3,\spin_0)
\end{align*}
where $\HMr_*(S^1\times \Sigma_3,\spin_0)=\mathbb{Q}_{-4}$ and the indices denote the absolute grading.
\end{cor}
\begin{remark}
The conventions for absolute gradings in monopole Floer homology differ from those in Heegaard Floer by shifting down by $b_1(Y)/2$.
\end{remark}

\begin{proof}
We begin with the case $g=2$, and set $\delta>0$. Consider a basis of $H_1(\mathbb{T})=H^1(S^1)\oplus H^1(\Sigma_2)$ of the form $z,x_1,y_1,x_2,y_2$ where the latter is symplectic for the cup product on $H^1(\Sigma_2)$. The triple cup product is then given by
\begin{equation*}
\cup_Y^3 =z\cup x_1\cup y_1+z\cup x_2\cup y_2.
\end{equation*}
Theorem \ref{simplestcomp} then identifies $\HMt_*(S^1\times \Sigma_2,\spin_0)$ with the homology of the following complex:
\begin{center}
\begin{tikzcd}[column sep=2em, row sep=0.5em]
	{\mathcal{T}_+} \\
	& {\mathcal{T}_+^{\oplus 5}} && {\underline{\mathcal{T}_+}} \\
	&& {\mathcal{T}_+^{\oplus 10}} && {\underline{\mathcal{T}_+^{\oplus 4}}} \\
	&&& {\mathcal{T}_+^{\oplus 9}} && {\underline{\mathcal{T}_+}} \\
	&&&& {\mathcal{T}_+}	
	\arrow[two heads, from=1-1, to=2-4]
	\arrow[two heads, from=2-2, to=3-5]
	\arrow[two heads, from=3-3, to=4-6]
\end{tikzcd}
\end{center}
The underlined towers correspond to $H_*(\mathbb{T}_-)$ because by Lemma \ref{perttheta} the subset $\mathbb{T}_-$ is a disk bundle over $\Sigma_2$, the latter being the theta divisor; the shift comes from the spectral flow. Here the the $i$th column correspond to Morse index filtration $5-i$ on each of the pieces of $\mathbb{T}\setminus \mathsf{K}$. The fact that the arrows are surjections readily follows from the description of the triple cup product and the computation in $\HMb_*$.
\par
In the genus three case the diagram looks instead as the following:
\begin{center}
\begin{tikzcd}[column sep=2em, row sep=0.5em]
	{\mathcal{T}_+} \\
	& {\mathcal{T}_+^{\oplus7}} && {\underline{\mathcal{T}_+}} \\
	&& {\mathcal{T}_+^{\oplus 21}} && {\underline{\mathcal{T}_+}^6} \\
	&&& {\mathcal{T}_+^{\oplus 34}} && {\underline{\mathcal{T}_+}^{15}\oplus\underline{\widetilde{\mathcal{T}_+}}} \\
	&&&& {\mathcal{T}_+^{\oplus29}\oplus\widetilde{\mathcal{T}_+}} && {\underline{\mathcal{T}_+}^6} \\
	&&&&& {\mathcal{T}_+^{\oplus 6}} && {\underline{\mathcal{T}_+}} \\
	&&&&&& {\mathcal{T}_+^{}}
	\arrow[two heads, from=1-1, to=2-4]
	\arrow[two heads, from=2-2, to=3-5]
	\arrow[two heads, from=3-3, to=4-6]
	\arrow[two heads, from=4-4, to=5-7]
	\arrow[two heads, from=5-5, to=6-8]
	\arrow[dashed, from=5-5, to=4-6]
\end{tikzcd}
\end{center}
where the solid arrows involve only the towers without a tilde. The computation is essentially the same as the genus $2$ case except for the additional towers $\underline{\widetilde{\mathcal{T}_+}}$ and ${\widetilde{\mathcal{T}_+}}$, which correspond to the one dimensional kernel of $H_*(\mathbb{T}_-)\rightarrow H_*(\mathbb{T})$ arising from the extra $\mathbb{Q}$ summand in $\mathrm{Sym}^2(\Sigma_3)$, see (\ref{sym3}). The dashed map is of the form (\ref{dashed}) and its kernel gives rise to the reduced homology $\mathbb{Q}$.
\par
Finally, absolute gradings can be determined by noticing that the manifolds admit an orientation reversing diffeomorphism and the duality isomorphism
\begin{equation*}
\HMt_*(-Y)\rightarrow \HMf^*(Y)
\end{equation*}
maps elements of degree $j$ to elements of degree $-1-b_1(Y)-j$, see \cite[Ch. 28]{KM}.
\end{proof}

\begin{remark}
It is natural to ask whether our approach applies to higher genus surfaces; there are two aspects of the story related to spectral and algebraic geometry respectively.
\par
On the spectral side, there are some known examples of surfaces of genus with $\lambda_1^*>1$. For example, the Bring curve (the common zero locus of the equations $v^k+w^k+x^k+y^k+z^k$ for $k=1,2,3$) has genus $4$ and $\lambda_1\approx 1.92$ \cite{Coo} while the Fricke-Macbeath surface, the unique Hurwitz surface of genus $7$, satisfies $\lambda_1\in[1.23,1.26]$ \cite{Lee}. Notice though that for genus high enough there are no hyperbolic surfaces with $\lambda_1^*>1$ (see for example \cite{FP}). Indeed, from the point of view of spectral theory of the Laplacian of functions on a hyperbolic surface it is natural to divide eigenvalues in small or large, the threshold being $1/4$ (significance of the latter is that it is the spectral gap in $L^2$ of the hyperbolic plane $\mathbb{H}^2$).
\par
The main complication arises on the algebro-geometric side, as the theta divisor is \textit{never} smooth when $g\geq 4$ \cite{ACGH}. Smoothness was a key input in proving that our our examples are of the simplest type in Lemma \ref{perttheta}. For example, for a generic genus $4$ surface, there are exactly two singular points in $\Theta$ consisting of line bundles with $h^0=2$; furthermore the singular point $\Theta\subset\mathrm{Jac}(\Sigma)$ is modeled after the singular quadric
\begin{equation*}
\{xy=tz\}\subset\mathbb{C}^4
\end{equation*} 
and the Abel-Jacobi map $\mathcal{A}$ is given by the small resolution.

\end{remark}

\vspace{0.3cm}
\section{The three-torus}\label{T3}

We will prove our main results, Theorem \ref{main} and \ref{mainspin}. We begin by briefly recalling how to determine of the Floer homology group
\begin{equation}
\HMt_*(T^3,\spin_0)=\mathcal{T}_+^{\oplus3}\oplus\mathcal{T}_+^{\oplus3}\langle-1\rangle
\end{equation}
by looking at a flat metric (cf. \cite[Ch. $38$]{KM}). There is a unique flat spin$^c$ connection $B_0$ for which $B_0$ has kernel, corresponding to the two dimensional space of parallel spinors. By adding a $\delta$-perturbation to the equations, this gives rise to a small $S^2\subset\mathbb{T}_{T^3}$ of operators having kernel as in Lemma \ref{perttheta}. Choosing a standard Morse function on $\mathbb{T}_{T^3}$ so that minimum lies inside $S^2$, the Floer chain complex is then directly determined to be
\begin{center}
\begin{tikzcd}[column sep=2em, row sep=0.5em]
	{\mathcal{T}_+} \\
	& {\mathcal{T}_+^{\oplus3}} && {\underline{\mathcal{T}_+}} \\
	&& {\mathcal{T}_+^{\oplus3}}
	\arrow[from=1-1, to=2-4]
\end{tikzcd}
\end{center}
where the underlined summand is shifted up because of the spectral flow.
\\
\par
There are a few special features when discussing a spectrally large manifold $Y$ with $b_1=3$. First of all, for a generic spinorial perturbation $A$, $\mathsf{K}$ is smooth. Of course, it will in general not be of the simplest type. For a component $U\subset \mathbb{T}_Y\setminus \mathsf{K}$, so that $U$ is a component of $\mathrm{U}\subset \mathrm{int}\mathbb{T}_j\setminus\mathbb{T}_{j-1}$ for some $j$, we call its upper boundary the intersection of $\bar{U}\cap\partial\mathbb{T}_j$ and its lower boundary the intersection of $\bar{U}\cap\partial\mathbb{T}_{j-1}$. We can then find an $A$-transverse Morse function $f:\mathbb{T}_Y\rightarrow \mathbb{R}$ satisfying the following properties:
\begin{enumerate}[(a)]
\item $f$ has maxima only in the components of $\mathbb{T}_Y\setminus\mathsf{K}$ with no upper boundary, and has exactly one maximum in each of them; similarly, $f$ has minima only in the components of $\mathbb{T}_Y\setminus\mathsf{K}$ with no lower boundary, and has exactly one maximum in each of them. 
\item $f$ achieves the same value $M$ at all maxima, and the same value $m$ at all minima.
\end{enumerate}
Condition (a) can be achieved by suitable canceling $0$-$1$ and $2$-$3$ handle pairs. After that, one can modify the values of the maxima and the minima (without changing them outside a neighborhood) to achieve (b). Condition (b) is useful because of the following special feature of $b_1=3$.
\begin{lemma}\label{morseind}
Suppose an $A$-adapted Morse function $f$ satisfies (b). Then the Morse index of the underlying critical point induces a filtration on $\check{C}_*$. Furthermore, on Morse index $1$ critical points $\check{\partial}=\bar{\partial}^s_s$, and on index $0$ critical points $\check{\partial}=0$.
\end{lemma}
\begin{proof}
The main observation is that the map $\partial^u_s$ appearing in $\check{\partial}$ preserves the parity of the underlying Morse index. Furthermore, it also strictly lowers the value of the Morse function and one concludes by a quick analysis of the effect of $\partial^u_s\bar{\partial}^s_u$ on points with Morse index $0\leq i\leq 3$.
\end{proof}

We are now ready to prove our main results.

\begin{proof}[Proof of Theorem \ref{main}]
Fix an $A$-transverse Morse function satisfying (a) and (b). By Lemma \ref{morseind} we can consider the spectral sequence induced by the Morse filtration. Let us denote
\begin{equation*}
E^2=\mathcal{F}_3\oplus \mathcal{F}_2\oplus \mathcal{F}_1\oplus \mathcal{F}_0
\end{equation*}
where the subscripts denote the filtration index. Notice that $\mathcal{F}_i$ lies in degree $i$ modulo $2$, so that $d_2=0$ and $E^2=E^3$. In degrees high enough, $(E^3,d_3)$ is forced to look like the Floer chain complex of the flat torus described above. In particular we have
\begin{equation*}
\mathcal{F}_i=
\begin{cases}
\mathcal{T}_+\langle d_i\rangle\oplus G_i \text{ if }i=0,3\\
\mathcal{T}_+^{\oplus3}\langle d_i\rangle\oplus G_i \text{ if }i=1,2
\end{cases}
\end{equation*}
for some integer $d_i$ and torsion $\mathbb{Q}[U]$-modules $G_i$.
\par
Consider collection of all minima $x_i$ for $i=1,\dots,n$ of ${f}$. Because $\mathbb{T}$ is connected, we can find (up to reordering) critical points $y_i$ for $1\leq i\leq n-1$ of index $1$ such that $\partial y_i=x_i-x_{i+1}$. Assumption (a) then assures that there is non-zero spectral flow on each of the trajectories between $y_i$ and $x_i$ or $x_{i+1}$. Hence, using Lemma \ref{morseind}, these critical points determine a subcomplex in $\check{C}_*$ of the form
\begin{equation}\label{01}
\begin{tikzcd}[column sep=tiny]
&\mathcal{T}_+ \arrow[dl]\arrow[dr]&&\cdots&&&\mathcal{T}_+\arrow[dr] \arrow[dl]&\\
     \mathcal{T}_+&&\mathcal{T}_+&\cdots&&\mathcal{T}_+&&\mathcal{T}_+
\end{tikzcd}
\end{equation}
where the top (resp. bottom) row consists of the $y_i$ (resp $x_i$), and each arrow is multiplication by $U^{k}$ for some $k\geq 1$ because of the spectral flow (we are not keeping track of gradings in the figure). From here we see that in fact
\begin{equation}\label{F0}
\mathcal{F}_0=\mathcal{T}_+\langle d_0\rangle
\end{equation}
i.e. $G_0=0$, and $G_1\neq0$ as soon as there is more than one minimum (cf. the analogous computations for the case $b_1=1$ in \cite{LinTheta}). Considering now the higher differential
\begin{equation}\label{d3}
d_3:\mathcal{F}_3\rightarrow \mathcal{F}_0,
\end{equation}
which is clearly surjective because of the $\mathbb{Q}[U]$-module structure, we see that it is be injective. This follows because anything in the kernel would lead to non-trivial reduced homology in $\HMt_*(T^3,\spin_0)$, again because the $U$-action from towers lowers the filtration level. Hence we conclude that
\begin{equation}\label{F3}
\mathcal{F}_3=\mathcal{T}_+\langle d_0+1\rangle.
\end{equation}
Turning upside down the argument for critical points of Morse index $0$ and $1$ above (\ref{01}), this implies that $f$ has exactly one maximum. Furthermore, for grading reasons (\ref{d3}) is the only non-zero component of $d_3$. In particular, $G_1$ survives in the $E^{\infty}$ page. Because of filtration reasons, it contributes to reduced homology, hence it vanishes, and $f$ has exactly one minimum as well.
\par
Furthermore, comparing the gradings in (\ref{F0}) and (\ref{F3}), we see that any path between the maximum and the minimum has spectral flow $-1$. This implies in particular that $\mathsf{K}$ is of the simplest type, so that we can apply Theorem \ref{simplestcomp}. Because $T^3$ has trivial reduced Floer homology, we see that $i_*:H_*(\mathbb{T}_-)\rightarrow H_*(\mathbb{T})$ is injective. By inspecting the compatibility with the $\Lambda^* H_1(T^3)$-action, the $E^1$ page of the simplest type spectral sequence is forced to be either that of the flat torus described above or
\begin{center}
\begin{tikzcd}[column sep=2em, row sep=0.5em]
	& {{\mathcal{T}}_+^{\oplus3}} &&  \\
	{\underline{\mathcal{T}}_+} && {\mathcal{T}_+^{\oplus3}}
	\\&&& {{\mathcal{T}_+}}
	
	\arrow[from=2-1, to=3-4]
\end{tikzcd}
\end{center}
which is simply the chain complex corresponding to $(-\delta,-\varepsilon)$. Let us focus on the former as the latter case is identical. We see that $H_*(\mathbb{T}_-)=H_*(\mathrm{pt})$. This implies that $\mathsf{K}$ is a union of spheres, because for a general three-manifold $M$ the inclusion of the boundary $\partial M\subset M$ satisfies the `half dies, half lives' property, i.e. the map induced by the inclusion
\begin{equation*}
H_1(\partial M;\mathbb{Q})\rightarrow H_1(M;\mathbb{Q})
\end{equation*}
has kernel of dimension $b_1(\partial M)/2$. Finally, there is only one boundary component because otherwise $H_2(\mathbb{T}_-)=H^1(\mathbb{T_-},\mathsf{K})$ would be non trivial; hence $\mathsf{K}$ is a two-sphere.
\end{proof}

\begin{proof}[Proof of Theorem \ref{mainspin}]
For a generic small function $\tau$, none of the perturbed Dirac operators $D_B-\tau$ where $B\in\mathbb{T}$ corresponds to one of the eight spin structures will have kernel. Now, in dimension three the spin Dirac operators are quaternionic linear, and the perturbations above are quaternionic as well. Because of this, there is a well defined value of the spectral flow between them modulo two. In \cite{LinRok} it is shown that this value is precisely the difference of the corresponding Rokhlin invariants.
\par
The additional perturbation $A_0$ is not quaternionic linear; but provided its norm is small enough so that none of $D_B-\tau-tA_0$ for $B$ spin and $t\in[0,1]$ has kernel, it will not change the value of the spectral flow. In particular, two points $B,B'$ corresponding to spin structures will be on the same component of $\mathbb{T}\setminus \mathsf{K}$ exactly when they have the same Rokhlin invariant. Finally, the component which is a ball contains $\spin_0$ by examining the proof above.
\end{proof}

\vspace{0.3cm}

\bibliographystyle{alpha}
\bibliography{biblio}

\end{document}